\documentclass[a4paper]{amsart}
\usepackage[T1]{fontenc}
\usepackage{graphicx}
\usepackage{verbatim}
\usepackage{tikz} 
\usepackage{enumerate}
\usepackage{pdfsync}
\usetikzlibrary{calc} 

\newcommand{\R}{{\mathbb R}}
\newcommand{\E}{{\mathbb E}}

\newcommand{\PP}{{\mathbb P}}

\newcommand{\X}{{\mathcal X}}

\newcommand{\N}{{\mathbb N}}

\newcommand{\F}{{\mathcal F}}

\newcommand{\GNZ}{{Georgii-Nguyen-Zessin }}
\newcommand{\1}[1]{\mathbf{1}_{ {#1} }}

\newcommand{\dom}{{\rm Dom}}

\renewcommand{\d}{{\text{d}}}
\newcommand{\moins}{\,\backslash \, }

\theoremstyle{remark}
\newtheorem*{rem}{Remark}

\theoremstyle{definition}

\theoremstyle{plain}
\newtheorem{thm}{Theorem}[section]
\newtheorem{lem}{Lemma}[section]
\newtheorem{cor}{Corollary}[section]
\newtheorem{prop}{Proposition}[section]

\theoremstyle{definition}
\newtheorem*{ex}{Example}
\newtheorem{hyp}{Hypothesis}
\newtheorem{defin}{Definition}

\title{Moment formulae for general point processes}
\author{L. Decreusefond and I. Flint}
\address{Institut Telecom, Telecom Paristech, CNRS LTCI, Paris, France}
\email{\{laurent.decreusefond,ian.flint\} @telecom-paristech.fr}
\date{}

\begin{document}

\begin{abstract}
	The goal of this paper is to generalize most of the moment formulae obtained in \cite{Pa}. More precisely, we consider a general point process $\mu$, and show that the relevant quantities to our problem are the so-called Papangelou intensities. Then, we show some general formulae to recover the moment of order $n$ of the stochastic integral of a random process. We will use these extended results to study a random transformation of the point process.
\end{abstract}

\maketitle

\section{Introduction}

Point processes constitute a general framework used to model a wide variety of phenomena. The underlying theory is well understood, and the relevant literature is abundant (see \cite{DVa} and \cite{Ka} for example). However, we have a much deeper understanding of the Poisson point process, which is one of the reasons for its use in a lot of practical cases. In particular, one has a chaos-expansion of Poisson functionals, concentration inequalities, moment formulae, etc. (\cite{Pa,Pb}) On the other hand, one lacks most of these tools for more general point processes. Our goal is to obtain moment formulae for very general point processes that are only required to have a Papangelou intensity. Intuitively, the Papangerlou intensity $c$ is such that $c(x,\xi) \, \lambda(\d x)$ is the conditional probability of finding a particle in the vicinity of $x$, given the configuration $\xi$ (here $\lambda$ is the reference measure). Historically, the first type of processes satisfying this condition is the Gibbs process. For a Gibbs process, $c(x,\xi) = e^{H(x\cup \xi) - H(\xi)}$, where $H$ is a global energy function, chosen in a suitable class of functions. The interested reader can find further information regarding Gibbs processes in \cite{Ra, Rb} as well as in \cite{XZa}. More recently, determinantal processes have been found to have Papangelou intensities under certain conditions (see \cite{GYa} as well as \cite{Yb}).

As we have stated previously, our aim is to obtain moment formulae for quite general point processes. Hence, we will follow the path of \cite{Pc}, in which some moment formulae were obtained for the Poisson point process (PPP). In particular, the main result of \cite{Pc} is the following formula, obtained in the case of the PPP:
\begin{align*}
		\E[  \big( \int u_x &(\xi) \,\xi(\d x) \big)^n ] \\
		&= \! \! \! \! \sum_{\{P_1,\dots,P_k\} \in \mathcal{T}_n} \! \!  \E \Big[ \int_{E^k} u_{x_1}^{|P_1|} \dots  u_{x_k}^{|P_k|} &(\xi \cup x_1 \cup \dots \cup x_k) \,\lambda(\d x_1)\dots \lambda(\d x_k)  \Big],
\end{align*}
where $\mathcal{T}_n$ is the set of all partitions of $\{1,\dots, n \}$, $|P_i|$ is the cardinality of $P_i$, $i=1,\dots,k$, and u : $E \times \X \rightarrow \R$ is a nonnegative measurable process.

The proofs in \cite{Pc} are mostly based on the use of previous results related to Malliavin calculus (in particular the formula that gives $\E [\delta(u)^n]$). In this paper, we generalize all the formulae in \cite{Pc} to the case of a point processes which has Papangelou intensities (which obviously includes the case of the PPP). Our proofs are mainly based on the \GNZ formula, and as a consequence, we also obtain analogues of the formula that gives $\E [\delta(u)^n]$. In all cases, the difference between the PPP and a general point process is a randomization of the underlying measure $\lambda$ obtained by multiplying it by $c(\cdot,\xi)$.

Our results also allow us to study random transformations of point processes. In the case of a general point process $\mu$, let us introduce a random transformation $\tau$, such that each particle $x$ of the configuration $\xi$ is moved to $\tau(x,\xi)$. Then, we obtain an explicit characterization of $\tau \mu$ if we only assume that $\tau$ is invertible and exvisible (the term will be defined precisely in section 5). An application is to show that the non-random transformation of a dererminantal measure yields another determinantal measure.

The remainder of the article is organized as follows. In section $2$, we introduce the basic tools that are used to study point processes. In section $3$, we give the main result of the paper, as well as most interesting consequences. In part $4$, we show analogue formulae for $\E [\delta(u)^n]$, where $\delta$ is the divergence, which will be rigorously defined in the case of a general point process. To conclude, section $5$ will deal with the study of a random transformation of the measure $\mu$.

\section{Notations and general results}

\subsection{Point processes}

Let $E$ be a Polish space, $\mathcal{O}(E)$ the family of all non-empty open subsets of $E$ and $\mathcal{B}$ denotes the corresponding Borel $\sigma$-algebra.  $\lambda$ is a Radon measure on $(E,\mathcal{B})$. Let $\X$ be the space of locally finite subsets in $E$, also called the configuration space:
\begin{equation*}
	\X = \{ \xi \subset E \,: \, | \Lambda \cap \xi | < \infty \, \text{ for any compact } \Lambda \subset E    \}.
\end{equation*}
In fact, $\X$ consists of all positive integer-valued Radon measures such that for all $x \in E$, $\xi({x}) \le 1$. Hence, it is naturally topologized by the vague topology, which is the weakest topology such that for all continuous and compactly supported functions $f$ on $E$, the mapping
\begin{equation*}
	\xi \mapsto \langle f, \xi \rangle := \sum_{x \in \xi} f(x)
\end{equation*}
is continuous. We denote by $\F$ the corresponding $\sigma$-algebra. We will call elements of $\X$ configurations and identify a locally finite configuration $\xi$ with the atomic Radon measure $\sum_{x \in \xi} \varepsilon_x$, where we have written $\varepsilon_x$ for the Dirac measure at $x \in E$. For a given $\xi = \sum_{x \in \xi} \varepsilon_x$, we will usually view $\xi$ as a set, and write $\xi \cup x_0 = \xi \cup \{ x_0 \} $ for the addition of a particle at $x_0$ and $\xi \moins x_0 = \xi \moins \{ x_0 \}$ for the removal of a particle at $x_0$. Moreover, for a measurable nonnegative $u : E \times \X \rightarrow \R$, we will often use the notation  $\int u_x(\xi) \,\xi(\d x) := \sum_{x \in \xi} u_x(\xi)$.

Next, let $\X_0 = \{ \xi \in \X \, : \, | \xi | < \infty \}$ the space of all finite configurations on $E$. $\X_0$ is naturally equipped with the trace $\sigma$-algebra $\F_0 = \F |_{\X_0}$. 

A random point process is defined as a probability measure on $(\X, \F)$. A random point process $\mu$ is characterized by its Laplace transform, which is defined for any measurable non-negative function $f$ on $E$ as
\begin{equation*}
	{\rm L}(f) = \int_\X e^{- \sum_{x \in \xi} f(x)} \,\mu(\d \xi)	.
\end{equation*} 
Now, let us introduce a number of measures useful in the study of point processes. Our notations are inspired by that of \cite{GYa}, where the reader can also find a brief summary of many properties of Papangelou intensities.
\begin{defin}
	We define the $\lambda$-sample measure $L$ on $(\X_0, \F_0)$ by the identity
	\begin{equation*}
	\int f(\alpha) \, L(\d \alpha) = \sum_{n \ge 0} \frac{1}{n!} \int_{E^n} f(\{ x_1, \dots, x_n \}) \,\lambda(\d x_1) \dots  \lambda(\d x_n),
	\end{equation*}
	for any measurable nonnegative function $f$ on $\X_0$.
\end{defin}
Point processes are often characterized via their correlation function, defined as below.
\begin{defin}[Correlation function]
	A point process $\mu$ is said to have a correlation function $\rho : \X_0 \rightarrow \R$ if $\rho$ is measurable and 
	\begin{equation*}
		\int_\X \sum_{\alpha \subset \xi, \ \alpha \in \X_0} f ( \alpha) \,\mu(\d \xi) = \int_{\X_0} f ( \alpha )\, \rho(\alpha) \, L(\d \alpha) ,
	\end{equation*}
	for all measurable nonnegative functions $f$ on $\X_0$. For $\xi = \{ x_1, \dots, x_n \}$, we will sometimes write $\rho(\xi) = \rho_n(x_1,\dots,x_n)$ and call $\rho_n$ the $n$-th correlation function, where here $\rho_n$ is a symmetrical function on $E^n$. 
\end{defin}
It can be noted that correlation functions can also be defined by the following property, both characterizations being equivalent in the case of simple point processes.
\begin{prop}
	A point process $\mu$ is said to have correlation functions $(\rho_n)_{n \in \N}$ if for any $A_1,\dots,A_n$ disjoint bounded Borel subsets of $E$,
	\begin{equation*}
		\E[\prod_{i=1}^n \xi(A_i)] = \int_{A_1\times \dots \times A_n} \rho_n(x_1,\dots,x_n) \,\lambda( \d x_1)\dots \lambda(\d x_n)	.
	\end{equation*}
\end{prop}
Recall that $\rho_1$ is the mean density of particles with respect to $\lambda$, and 
\begin{equation*}
	\rho_n(x_1,\dots,x_n) \, \lambda(\d x_1)\dots \lambda(\d x_n)
\end{equation*}
is the probability of finding a particle in the vicinity of each $x_i$, $i=1,\dots,n$.

Let us now define the so-called Campbell measures:
\begin{defin}[Campbell measures]
	The reduced Campbell measure of a point process $\mu$ is the measure $C_\mu$ on the product space $(E \times \X, \mathcal{B} \otimes \F)$ defined by 
	\begin{equation*}
		C_\mu(A \times B) = \int \sum_{x \in \xi} \,\1A (x) \1B(\xi \moins x) \, \mu(\d \xi),
	\end{equation*}
	where $A \in \mathcal{B}$ and $B \in \F$.
	We define similarly the reduced compound Campbell measure of a point process $\mu$ is the measure $\hat{C}_\mu$ on the product space $(\X_0 \times \X, \F_0 \otimes \F)$ defined by 
	\begin{equation*}
		\hat{C}_\mu(A \times B) = \int \sum_{\alpha \subset \xi, \ \alpha \in \X_0} \1A (\alpha) \1B(\xi \moins \alpha) \, \mu(\d \xi),
	\end{equation*}
	where $A \in \F_0$ and $B \in \F$.
\end{defin}
The results in this paper can be obtained under the assumption $C_\mu \ll \mu$, but we will work under a slightly stronger condition in order to find more compact results. We will thus assume throughout this paper that the following condition is fulfilled:
\begin{hyp}[Condition $(\Sigma_\lambda)$]
\label{papangelou}
	The point process $\mu$ is assumed to satisfy condition $(\Sigma_\lambda)$, i.e. $C_\mu \ll \lambda \otimes \mu$.
\end{hyp}
Henceforth, any Radon-Nikodym density $c$ of $C_\mu$ relative to $\lambda \otimes \mu$ is called a version of the Papangelou intensity of $\mu$. 

The preceding assumption also implies that $\hat{C}_\mu \ll L \otimes \mu$ and we will thus similarly  denote any Radon-Nikodym density of $\hat{C}_\mu$ relative to $L \otimes \mu$ by $\hat{c}$. One then has for any $\xi \in \X$, $\hat{c}(\emptyset,\xi) = 1$, as well as for all $x\in E$, $\hat{c}(x,\xi) = c(x,\xi)$. Moreover, the following commutation relationship is also verified:
\begin{equation}
\label{roulement}
	\forall \xi, \eta, \nu \in \X, \quad \hat{c}(\nu,\eta \cup \xi)\, \hat{c}(\eta, \xi) = \hat{c}(\nu \cup \eta, \xi).
\end{equation}

Hypothesis \ref{papangelou}, along with the definition of the reduced Campbell measure, allows us to write the following important identity, known as the \GNZ identity:
\begin{equation}
\label{mecke}
	\int_\X \sum_{x \in \xi} u (x, \xi \moins x) \,\mu(\d \xi) = \int_\X \int_E u (x, \xi )\, c(x,\xi) \,\lambda(\d x) \mu(\d \xi) ,
\end{equation}
for all measurable nonnegative functions $u : E \times \X \rightarrow \R$. We also have a similar identity for all measurable nonnegative functions $u : \X_0 \times \X \rightarrow \R$:
\begin{equation}
\label{mecke2}
	\int_\X \sum_{\alpha \subset \xi, \ \alpha \in \X_0} u (\alpha, \xi \moins \alpha) \,\mu(\d \xi) = \int_\X \int_E u (\alpha, \xi )\, \hat{c}(\alpha,\xi) \,\lambda(\d x) \mu(\d \xi) ,
\end{equation}

Combining relation (\ref{mecke}) and the definition of the correlation functions, we find
\begin{equation*}
	\E[ c(x,\xi) ] = \rho_1(x),
\end{equation*}
for almost every $x \in E$. We also find more generally, using (\ref{mecke2}), that
\begin{equation}
\label{evc}
	\E[ \hat{c}(\alpha,\xi) ] = \rho(\alpha),
\end{equation}
for almost every $\alpha \in \X_0$.

Let us continue by giving some examples of point processes satisfying assumption \ref{papangelou}.
\begin{ex}[Poisson process]
	Let us start by considering the well-known Poisson process with intensity $z$ with respect to $\lambda$, noted $\pi^{z \d \lambda} $, defined for example via its Laplace transform as follows:
	\begin{equation*}
		{\rm L}_{\pi^{z \d \lambda}}(f) = e^{- \int (1-e^{-f(x)}) \, z(x) \lambda(\d x) }	.
	\end{equation*} 
	In that case, $(\Sigma_\lambda)$ is verified and we have $c(x,\xi) = z(x)$, for any $x \in E$ and $\xi \in \X$, and the \GNZ identity reads
	\begin{equation*}
	\int_\X \sum_{x \in \xi} u_x ( \xi \moins x) \,\pi^{z \d \lambda}(\d \xi) = \int_\X \int_E u_x ( \xi )\, z(x) \,\lambda(\d x) \pi^{z \d \lambda}(\d \xi) ,
\end{equation*}
	which is a well known result in the Poisson case (known as the Campbell-Mecke identity). Recall that the Poisson measure has a correlation function defined by
	\begin{equation*}
		\rho_{\pi^{z \d \lambda}} (\xi) = \prod_{x \in \xi} z(x), \quad \xi \in \X.
	\end{equation*}
	One can notice that $\rho_1(x) = z(x)$, $x \in E$, which is also equal to $c(x,\xi)$. Indeed, recall that $c(x,\xi) \, \lambda(\d x)$ is interpreted as the probability of finding a particle in the differential region $\d x$, conditionally on $\xi$. For the Poisson process, the conditioning by $\xi$ does not add any information and thus the probability is also equal to $\rho_1(x) \, \lambda(\d x)$.
\end{ex}

\section{Moment formulae}

Let us start by proving a simple combinatorial lemma.
\begin{lem}
\label{lem1}
	Let $F$ be a function from the power set $\mathcal{P}(\N)$ to $\R$. Then for $n \in \N^*$,
	\begin{multline*}
		\sum_{k=1}^{n+1} \sum_{\mathcal{P} \in \mathcal{T}_{n+1}^k } F(\mathcal{P}) = \sum_{k=1}^{n}\sum_{\mathcal{P} =\{P_1,\dots,P_k \} \in \mathcal{T}_{n}^k }   F \big(\mathcal{P} \cup \big\{ \{ n+1 \} \big\} \big) \\
		\shoveright{+ \sum_{k=1}^{n}\sum_{\mathcal{P} =\{P_1,\dots,P_k \} \in \mathcal{T}_{n}^k }  \sum_{l=1}^k F \big(P_1,\dots,P_{l-1},P_l \cup \{n+1 \},P_{l+1},\dots,P_k \big).}\\
	\end{multline*}
\end{lem}
\begin{proof}

	Let us consider the functions $(\Xi_k)_{1 \le k \le n+1}$ defined as follows.
	\begin{equation*}
		\left.  \begin{array}{l l l}
			\Xi_k : &\mathcal{T}_{n+1}^k &\longrightarrow \mathcal{T}_{n}^k \cup \mathcal{T}_n^{k-1}	,\\ 
			&\{ P_1,\dots,P_k \} &\longmapsto \{ P_1,\dots,P_{l-1},P_l \moins \{ n+1 \},P_{l+1},\dots,P_k \}	,
		\end{array}\right.
	\end{equation*}
	where $1 \le l \le k$ is such that $\{ n+1 \} \in P_l$. Moreover, let us define the function $\Xi$ as
	\begin{equation*}
		\left.  \begin{array}{l l l}
			\Xi : &\mathcal{T}_{n+1} &\longrightarrow \mathcal{T}_{n} ,\\ 
			&\mathcal{P} &\longmapsto \Xi_{| \mathcal{P} |}(\mathcal{P}),
		\end{array}\right.
	\end{equation*}
	where $\mathcal{T}_{n}$ is the set of all partitions of $\{ 1,\dots,n \}$. Then, $\Xi$ is surjective from $ \mathcal{T}_{n+1}$ into $\mathcal{T}_{n}$ and moreover, for $\mathcal{P} = \{ P_1,\dots,P_{| \mathcal{P} |}  \} \in \mathcal{T}_n$, we have that
	\begin{equation*}
		\Xi^{-1} (\mathcal{P} ) = \big\{ \mathcal{P} \cup \big\{ \{ n+1 \} \big\} \big\} \cup \bigcup_{l=1}^{| \mathcal{P} |}  \big\{ P_1,\dots,P_l \cup \{ n+1 \},\dots,P_{| \mathcal{P} |} \big\}.
	\end{equation*}
	Then, using the preceding observations, we obtain
	\begin{align*}
		\sum_{k=1}^{n+1} \sum_{\mathcal{P} \in \mathcal{T}_{n+1}^k } F(\mathcal{P}) &= \sum_{\mathcal{P} \in \mathcal{T}_{n+1} } F(\mathcal{P}) \\
		&= \! \! \! \! \! \! \! \! \! \! \! \!  \sum_{\substack{ \mathcal{P} \in \mathcal{T}_{n} \\ \mathcal{P} =  \{ P_1,\dots,P_{| \mathcal{P} |}  \}  } } \! \! \! \! \! \! \! \! \! \! \! \!  F(\mathcal{P} \cup \{ n+1 \}) + \sum_{l=1}^{| \mathcal{P} |} F( \big\{ P_1,\dots,P_l \cup \{ n+1 \},\dots,P_{| \mathcal{P} |} \big\}),
	\end{align*}
	which is the desired result once we sum on the different lengths possible for elements of $ \mathcal{T}_{n}$.
\end{proof}

We can now give the first important result of this paper, which will yield many cases of particular interest.
\begin{thm}
\label{prop1}
	For any $n \in \N$, any measurable nonnegative functions $u_k : E \times \X \rightarrow \R$, $k = 1,\dots,n$, and any bounded function $F$ on $\X$, we have
	\begin{align*}
		\E[ F \prod_{k=1}^n \int u_k(y,\xi) \,\xi(\d y) ] = \sum_{k=1}^n  \sum_{\mathcal{P} \in \mathcal{T}_n^k}  \E \Big[ \int_{E^k} & F(\xi \cup x) u^\mathcal{P}(x,\xi \cup x) \hat{c}(x,\xi) \,\lambda_k(\d x) \Big] ,
	\end{align*}
	where $\mathcal{T}_n^k$ is the set of all partitions of $\{1,\dots, n \}$ into $k$ subsets. Here, for $\mathcal{P} = \{ P_1,\dots,P_k \} \in \mathcal{T}_n^k$, we have used the compact notation $x = (x_1,\dots,x_k)$, as well as $\lambda_k(\d x) = \lambda(\d x_1)\dots \lambda(\d x_k) $ and 
	\begin{equation*}
		u^\mathcal{P}(x, \xi) = \prod_{l=1}^k \prod_{i \in P_l} u_i(x_l,\xi).
	\end{equation*}
\end{thm}
\begin{proof}
	We will prove this result by induction on $n$. Let $u : E \times \X \rightarrow \R$ be a measurable nonnegative function. First note that by the \GNZ identity (\ref{mecke}), we have
	\begin{equation*}
		\E[ F(\xi) \int u(y,\xi) \,\xi(\d y) ] = \E[ \int F(\xi \cup x) u(x,\xi \cup x) \, c(x,\xi) \,\lambda(\d x) ].
	\end{equation*}
	Now, for $n\in \N^*$, let $u_1,\dots,u_{n+1}$ be nonnegative measurable functions on $E \times \X$. We have by induction:
	\begin{align*}
		&\E[ F \prod_{k=1}^{n+1} \int u_k(y, \xi) \,\xi(\d y) ] \\
		&= \E[ F(\xi)  \prod_{k=1}^n \int u_k(z, \xi) \big( \int u_{n+1} (y,\xi) \,\xi(\d y)\big)^{1/n} \,\xi(\d z)  ] \\
		&=  \sum_{k=1}^n  \sum_{\mathcal{P} \in \mathcal{T}_n^k} \E \Big[ \int_{E^k} F(\xi \cup x) u^\mathcal{P}\! (x,\xi \cup x) \hat{c}(x,\xi) \int u_{n+1}(z, \xi \cup x) \{\xi \cup x \} (\d z) \lambda_k(\d x) \Big].
	\end{align*}
	Here, the last part can be rewritten 
	\begin{equation*}
		 \int u_{n+1}(z, \xi \cup x) \{\xi \cup x \} (\d z) = \int u_{n+1}(z, \xi \cup x) \,\xi(\d z)  + \sum_{l =1}^k u_{n+1}(x_l,\xi \cup x).
	\end{equation*}
	Hence, after regrouping the terms, we find
	\begin{align*}
		&\E[ F \prod_{k=1}^{n+1} \int u_k(y, \xi) \,\xi(\d y) ]  \\
		&= \sum_{k=1}^n  \sum_{\mathcal{P} \in \mathcal{T}_n^k} \sum_{l=1 }^k \, \E \Big[ \int_{E^k} F(\xi \cup x) u^\mathcal{P}\! (x,\xi \cup x) u_{n+1} (x_l, \xi \cup x)\ \hat{c}(x,\xi) \,\lambda_k(\d x) \Big] \\
		&+ \sum_{k=1}^n  \sum_{\mathcal{P} \in \mathcal{T}_n^k} \E \Big[ \int_{E^k} F(\xi \cup x) u^\mathcal{P} (x,\xi \cup x) \big( \int u_{n+1}(z, \xi \cup x) \,\xi(\d z) \big) \, \hat{c}(x,\xi) \,\lambda_k(\d x) \Big] .
	\end{align*}
	Then, by Fubini's theorem, and the \GNZ formula (\ref{mecke}), the second sum is equal to
	\begin{align*}
		& \E \Big[ \int_{E^k} F(\xi \cup x) u^\mathcal{P} (x,\xi \cup x) \big( \int u_{n+1}(z, \xi \cup x) \,\xi(\d z) \big) \, \hat{c}(x,\xi) \,\lambda_k(\d x) \Big] \\
		&= \! \E \Big[ \! \int_{E^{k+1}} \! \! \! \! \! \! \! \! F(\xi \cup x \cup z) u^\mathcal{P}\! (x,\xi \cup x \cup z)  u_{n+1}(z, \xi \cup x \cup z) \hat{c}(x,\xi \cup z) c(z,\xi) \lambda_k(\d x) \lambda(\d z) \! \Big] \\
		&= \E \Big[ \int_{E^{k+1}} F(\xi \cup x \cup z) u^\mathcal{P} (x,\xi \cup x \cup z)  u_{n+1}(z, \xi \cup x \cup z)\, \hat{c}(x \cup z,\xi) \,\lambda_k(\d x) \lambda(\d z) \Big]\\
		&= \E \Big[ \int_{E^{k+1}} F(\xi \cup x) u^{\mathcal{P}\cup \{n+1\}} (x,\xi \cup x) \, \hat{c}(x,\xi) \,\lambda_{k+1}(\d x)  \Big],
	\end{align*}
	since by (\ref{roulement}), we know that $\hat{c}(\{x_1,\dots,x_k \},\xi \cup y) \, c(y,\xi) = \hat{c}(\{x_1,\hdots,x_k,y\},\xi )$. To summarize, we have found
	\begin{multline*}
		{ \E[ F \prod_{k=1}^{n+1} \int u_k(y, \xi) \,\xi(\d y) ] } \\
		{=  \sum_{k=1}^n  \sum_{\mathcal{P} \in \mathcal{T}_n^k} \Big( \sum_{l=1 }^k \, \E \Big[ \int_{E^k} F(\xi \cup x) u^\mathcal{P}\! (x,\xi \cup x) u_{n+1} (x_l, \xi \cup x)\ \hat{c}(x,\xi) \,\lambda_k(\d x) \Big] } \\
		\shoveright{+ \E \Big[ \int_{E^{k+1}} F(\xi \cup x) u^{\mathcal{P}\cup \{n+1\}} (x,\xi \cup x) \, \hat{c}(x,\xi) \,\lambda_{k+1}(\d x)  \Big] \Big),} \\
	\end{multline*}
	and we obtain the desired result by applying lemma \ref{lem1}.
\end{proof}
The previous quite general property includes many interesting cases. In particular, if all $u_k$, $k=1,\dots,n$ are equal, one generalizes the results in \cite{Pa} and obtains the moments of $\int u_x(\xi)\,\xi(\d x)$ :
\begin{cor}
	For any $n \in \N$, any measurable nonnegative function $u : E \times \X \rightarrow \R$; and any bounded function $F$ on $\X$, we have
	\begin{multline*}
		\E[ F \big( \int u(y,\xi) \,\xi(\d y) \big)^n ] = \sum_{k=1}^n  \sum_{\{P_1,\dots,P_k \} \in \mathcal{T}_n^k}  \E \Big[ \int_{E^k} F (\xi \cup x_1 \cup \dots \cup x_k)\\
			\shoveright{\prod_{l=1}^k u^{|P_l|} (x_l, \xi \cup x_1 \cup \dots \cup x_k )  \hat{c}(\{x_1,\dots,x_k \},\xi) \,\lambda(\d x_1)\dots \lambda(\d x_k)  \Big],}\\
	\end{multline*}
	where $|P_i|$ is the cardinality of $P_i$, $i=1,\dots,k$.
\end{cor}
This result includes the case where $u_x(\xi) = v(x)$ is a deterministic function:
\begin{cor}
	For any $n \in \N$, and any measurable nonnegative function $v$ on $E$, we have
	\begin{multline*}
		\E[ F \big( \int v(y) \,\xi(\d y) \big)^n ] =\sum_{k=1}^n  \sum_{\{ P_1,\dots,P_k \} \in \mathcal{T}_n^k} \int_{E^k} v(x_1)^{|P_1|} \dots v(x_k)^{|P_k|}\, \\
			\shoveright{ \E[ F(\xi \cup x_1 \cup \dots \cup x_k)\hat{c}(\{x_1,\dots,x_k \},\xi)] \,\lambda(\d x_1)\dots \lambda(\d x_k).}\\
	\end{multline*}
\end{cor}
The previous corollary yields
\begin{multline*}
	{\rm Cov} \big(F, (\int v(x) \,\xi(\d x) )^n \big) = \sum_{k=1}^n  \sum_{\{ P_1,\dots,P_k \} \in \mathcal{T}_n^k}  \int_{E^k} v(x_1)^{|P_1|} \dots v(x_k)^{|P_k|}\, \\
			\shoveright{\quad \E[ (F(\xi \cup x_1 \cup \dots \cup x_k) - F(\xi) ) \hat{c}(\{x_1,\dots,x_k \},\xi)] \,\lambda(\d x_1)\dots \lambda(\d x_k).}\\
\end{multline*}
The case of $F=1$ is also of particular interest:
\begin{cor}
\label{cor2}
	For any $n \in \N$, any measurable nonnegative non-random functions $v_k : E \rightarrow \R$, $k = 1,\dots,n$, we have
	\begin{multline*}
		\E[ \prod_{k=1}^n \int v_k(y) \,\xi(\d y) ] = \sum_{k=1}^n  \sum_{\{ P_1,\dots,P_k \} \in \mathcal{T}_n^k} \int_{E^k}  \prod_{i \in P_1} v_i(x_1) \dots \prod_{i \in P_k} v_i(x_k) \\
			\shoveright{\rho(\{x_1,\dots,x_k \}) \,\lambda(\d x_1)\dots \lambda(\d x_k).}\\
	\end{multline*}
\end{cor}
Note that we recover here a classical formula, which reads for $n=1,2,3$:
\begin{equation*}
	\E[  \int v(x) \,\xi(\d x) ] = \int_E v(x) \rho_1(x) \,\lambda(\d x),
\end{equation*}
\begin{equation*}
	\E[ \big( \int v(x) \,\xi(\d x)\big)^2 ] = \int_E v(x)^2 \rho_1(x) \,\lambda(\d x) + \int_{E^2} v(x)v(y)  \rho(\{x,y\}) \,\lambda(\d x)\lambda(\d y),
\end{equation*}
\begin{align*}
	\E[ \big( \int v(x) \,\xi(\d x)\big)^3 ] = \int_E v(x)^3 &\rho_1(x) \,\lambda(\d x) + 3\int_{E^2} v(x)^2 v(y)  \rho(\{x,y\}) \,\lambda(\d x)\lambda(\d y) \\
	&+ \int_{E^3} v(x) v(y) v(z)  \rho(\{x,y,z\}) \,\lambda(\d x)\lambda(\d y) \lambda(\d z).
\end{align*}
The previous corollary is interesting in itself because it is in fact an equivalent characterization of the existence of correlation functions. More precisely, we have the following result:
\begin{prop}
\label{charact}
	Let $(\rho_k)_{k \in \N}$ be a family of symmetrical, measurable functions, and $\rho_k : E^k \rightarrow \R$ for $k  \in \N$. Assume moreover that the measure $\mu$ is simple, in the sense that $\PP(\xi(x) \le 1)$ for all $x \in E$. Then, the measure $\mu$ possesses correlation functions $(\rho_k)_{k \in \N}$ (with respect to $\lambda$) if and only if, for any $n \in \N$, and any measurable nonnegative functions $v_k : E \rightarrow \R$, $k = 1,\dots,n$, we have
	\begin{multline*}
		\E[\, \prod_{k=1}^n \int v_k(x) \,\xi(\d x) ] =  \sum_{k=1}^n  \sum_{\{ P_1,\dots,P_k \} \in \mathcal{T}_n^k} \int_{E^k} \prod_{i \in P_1} v_i(x_1) \dots \prod_{i \in P_k} v_i(x_k) \\
			\shoveright{\rho_k(x_1,\dots,x_k) \,\lambda(\d x_1)\dots \lambda(\d x_k).}\\
	\end{multline*}
\end{prop}
\begin{proof}
	Assume that we have
	\begin{equation*}
		\int_\X \sum_{\alpha \subset \xi, \ \alpha \in \X_0} f ( \alpha) \,\mu(\d \xi) = \int_{\X_0} f ( \alpha )\, \rho(\alpha) \, L(\d \alpha) ,
	\end{equation*}
	where $f$ is any measurable nonnegative function on $\X_0$. Then, for $k \in \N$,
	\begin{equation*}
		\int_\X \sum_{\alpha \subset \xi, \ |\alpha| = k} f ( \alpha) \,\mu(\d \xi) = \frac{1}{k!} \int_{E^k} f (\{ x_1,\dots,x_k\} )\, \rho_k(x_1,\dots,x_k) \, \lambda(\d x_1)\dots \lambda(\d x_k).
	\end{equation*}
	Now, we can write for $n \in \N$ and $\xi \in \X$,
	\begin{equation*}
		\sum_{x_1,\dots,x_n \in \xi} \! \! \! v_1(x_1)\dots v_n (x_n) = n! \sum_{k=1}^n \sum_{\{ P_1,\dots,P_k\} \in \mathcal{T}_n^k} \sum_{\substack{\alpha \subset \xi,\ |\alpha | = k \\ \alpha = \{ x_1,\dots, x_k \}}} \!  \prod_{i \in P_1} v_i (x_1) \dots \! \prod_{i \in P_k}v_i (x_k),
	\end{equation*}
	where the $n!$ appears since when we write $\{x_1,\dots,x_k \} \subset \xi$, we only choose ordered subsets of $\xi$. Then, we find the desired result by taking the expectation of the previous equality. Indeed, we have by definition
	\begin{multline*}
		\E \big[ \sum_{\substack{\alpha \subset \xi,\ |\alpha | = k \\ \alpha = \{ x_1,\dots, x_k \}}}  \prod_{i \in P_1} v_i (x_1) \dots \prod_{i \in P_k}v_i (x_k) \big] \\
		\shoveright{= \frac{1}{n!}  \E \big[ \int_{E^k} \prod_{i \in P_1} v_i(x_1) \dots \prod_{i \in P_k} v_i(x_k)\hat{c} (\{ x_1,\dots,x_k \},\xi ) \,\lambda(\d x_1)\dots \lambda(\d x_k). \big],}\\
	\end{multline*}
	and we conclude by using (\ref{evc}).
	
	On the other hand, assume that there exists some symmetrical measurable functions $(\rho_k)_{k \in \N}$, such that for any $n \in \N$, and any measurable nonnegative functions $v^k : E \rightarrow \R$, $k = 1,\dots,n$, we have
	\begin{align*}
		\E[\, \prod_{k=1}^n \int v_k(y) \,\xi(\d y) ] =  \sum_{k=1}^n \sum_{\{ P_1,\dots,P_k\} \in \mathcal{T}_n^k}  \int_{E^k} & \prod_{i \in P_1} v_i(x_1) \dots \! \prod_{i \in P_k} v_i(x_k) \\
			&\rho_k(x_1,\dots,x_k ) \,\lambda(\d x_1)\dots \lambda(\d x_k).
	\end{align*}	
	Let $A_1,\dots,A_n$ be $n$ disjoint Borel subsets of $E$. Take $v^k = \1{A_k}$, $k=1,\dots,n$. Then, the different terms of the right-hand side sum are all equal to $0$, except for the subdivision consisting of only singletons. Hence,
	\begin{equation*}
		\E[\, \prod_{k=1}^n \xi(A_k) ] =   \int_{A_1 \times \dots \times A_n} \rho_k(x_1,\dots,x_n ) \,\lambda(\d x_1)\dots \lambda(\d x_n),
	\end{equation*}
	which means that $\mu$ has correlation functions $(\rho_k)_{k \in \N}$ since $\mu$ is a simple point process.
\end{proof}

\section{Extended moment formulae for the divergence}

The aim of this section is to obtain analogue formulae involving the divergence of a general operator. Let us now start by defining a new random measure $\nu$ by the formula
\begin{equation*}
	\int f(y,\xi) \,\nu(\d y) = \int f(y,\xi) \, \xi(\d y) - \int f(y,\xi) c(y,\xi) \, \lambda(\d y),
\end{equation*}
for $f$ on $E \times \X$ such that $\E [ \int_E | f(y,\xi) | c(y,\xi) \,\lambda(\d y)  ] < \infty$. 

In order to ensure the rest of the results in this section, we will need to make another assumption on $\mu$ which will be assumed to hold henceforth.
\begin{hyp}
\label{hypmoments}
	We assume that the moments of the point process $\mu$ are all finite, i.e. that for all $\varphi$ bounded, compactly supported function, and for all $n \in \N$,
	\begin{equation*}
		\E[ \big( \int \varphi(x)\,\xi(\d x) \big)^n ] < \infty.
	\end{equation*}
\end{hyp}
By corollary \ref{cor2} the previous condition can be rewritten as 
\begin{equation*}
	\forall n \in \N,\quad \int_{\Lambda^n} \rho_n(x_1,\dots,x_n) \,\lambda(\d x_1,\dots,\d x_n) < \infty,
\end{equation*}
for all compacts $\Lambda$ of $E$.

Then, we prove the following:
\begin{prop}
\label{prop2}
	We assume that hypothesis \ref{hypmoments} is verified. Then, for any $n \in \N$, any bounded process $u : E \times \X \rightarrow \R$ with compact support on $E$, and any bounded function $F$ on $\X$, we have
	\begin{multline*}
		\E[ F \big( \int u(y,\xi) \,\nu(\d y) \big)^n ] = \sum_{i=0}^n (-1)^i \binom{n}{i}  \sum_{k=1}^n  \sum_{\{P_1,\dots,P_k\} \in \mathcal{T}_{n-i}^k}  \\
		\shoveright{\E \Big[ \int_{E^k}\, \hat{c}(x,\xi) F (\xi \cup x) \prod_{l=1}^k u( x_l,\xi \cup x)^{| P_l |} \big( \int u(z,\xi \cup x) c(z,\xi \cup x) \,\lambda(\d z)  \big)^i \,\lambda_k(\d x) \Big].}\\
	\end{multline*}
\end{prop}
\begin{proof}
	This is a direct consequence of the binomial formula
	\begin{align*}
		\E[ &F \big( \int u(y,\xi) \,\nu(\d y) \big)^n ] \\
		&= \sum_{i=0}^n (-1)^i \binom{n}{i} \E \Big[ F \big( \int u(y,\xi) \,\xi(\d y) \big)^{n-i} \big( \int u(y,\xi)c(y,\xi)\,\lambda(\d y) \big)^i \Big],
	\end{align*}
	and thus we obtain the desired result by applying proposition \ref{prop1}.
\end{proof}
Recall that the associated formula for $\pi^1$ (the Poisson process of intensity $\lambda$), as obtained in \cite{Pa}, is
\begin{multline*}
		\E_{\pi^1}[ F \big( \int u(y,\xi) \,\nu(\d y) \big)^n ] \! = \! \! \sum_{i=0}^n (-1)^i \binom{n}{i} \! \!  \sum_{k=1}^n \! \sum_{\{P_1,\dots,P_k\} \in \mathcal{T}_{n-i}^k} \! \! \! \! \! \! \! \! \! \! \! \! \! \E_{\pi^1} \Big[ \int_{E^{i+k}}\! \! \! \! \! \! \! \! \! F(\xi \cup \{ x_1 \cup \dots \cup x_k \}) \\
	\shoveright{\prod_{l=1}^k u^{|P_l|} (x_l, \xi \cup x_1 \cup \dots \cup x_k ) \prod_{l=k+1}^{i+k} u(x_l,\xi \cup x_1 \cup \dots \cup x_k )\,\lambda(\d x_1)  \dots \lambda( \d x_{i+k})  \Big].}\\
\end{multline*}
Here, by comparison, the general formula can be written as 
\begin{multline*}
		\E[ F \big( \int u(y,\xi) \,\nu(\d y) \big)^n ]  \\
		\shoveright{=  \sum_{i=0}^n (-1)^i \binom{n}{i}  \sum_{k=1}^n \! \sum_{\{P_1,\dots,P_k\} \in \mathcal{T}_{n-i}^k} \! \! \! \! \! \! \! \!  \E \Big[ \int_{E^{i+k}} \! \! \! \! \! \! \!  G_k(\{x_1,\dots,x_{i+k}\},\xi)  F(\xi \cup \{ x_1 \cup \dots \cup x_k \})} \\
	\shoveright{\prod_{l=1}^k u^{|P_l|} (x_l, \xi \cup x_1 \cup \dots \cup x_k ) \prod_{l=k+1}^{i+k} u(x_l,\xi \cup x_1 \cup \dots \cup x_k )\,\lambda(\d x_1)  \dots \lambda( \d x_{i+k})  \Big].}\\
\end{multline*}
where
\begin{align*}
	G_k(\{x_1,&\dots,x_{i+k}\},\xi)  \\
	&= \hat{c}(\{x_1,\dots,x_k \},\xi) c(x_{i+1},\xi \cup x_1 \cup \dots \cup x_k)\dots c(x_{i+k},\xi \cup x_1 \cup \dots \cup x_k) \\
		&= \prod_{l=1}^k c(x_l , \xi \cup x_1 \cup \dots \cup x_l ) \prod_{l=i+1}^{i+k} c(x_l,\xi \cup x_1 \cup \dots \cup x_k).
\end{align*}

Let us now introduce the so-called divergence operator $\delta$. 
\begin{defin}[Divergence operator]
	For any measurable $u : E \times \X \rightarrow \R$ such that $\E[ \int | u(y,\xi)|  \,c(y,\xi)\,\lambda(dy) ] < \infty$, we define $\delta(u)$ as
\begin{equation*}
	\delta (u) = \int u(y,\xi \moins y) \, \nu(\d y) = \int u(y,\xi \moins y) \, \xi(\d y) - \int u(y,\xi) \, c(y,\xi) \,\lambda(\d y).
\end{equation*}
\end{defin}
We can notice that the divergence of a bounded process $u$ is correctly defined since for such a process,
\begin{equation*}
	\E[ \int | u(y,\xi)| \,c(y,\xi)\,\lambda(dy) ] \le C \, \E[ \int c(y,\xi)\,\lambda(dy) ] = 1.
\end{equation*}

The next proposition gives a moment formula for this newly introduced operator.
\begin{prop}
\label{prop4}
	For any $n \in \N$, any bounded process $u : E \times \X \rightarrow \R$, with compact support on $E$, and any bounded function $F$ on $\X$, we have
	\begin{multline*}
		\E[ F \big( \delta(u)  \big)^n ] \\
		\shoveright{ = \sum_{i=0}^n (-1)^i \binom{n}{i}  \sum_{k=1}^n \sum_{\{P_1,\dots,P_k\} \in \mathcal{T}_{n-i}^k }\! \! \! \! \! \! \!  \E \Big[ \int_{E^{i+k}} \! \! \! \! \! \!   G_k(\{x_1,\dots,x_{i+k}\},\xi)  F (\xi  \cup  \{ x_1, \dots, x_k \} )} \\
		\shoveright{\prod_{l=k+1}^{i+k} \! \! \! \! u (x_l, \xi  \cup  \{ x_1, \dots, x_k \} \! ) \prod_{l=1}^k  u^{|P_l |} (x_l, \xi \cup \{ x_1,\dots,\tilde{x_l},\dots,x_k\}   ) \lambda(\d x_1)  \dots \lambda( \d x_{i+k}) \Big],}\\
	\end{multline*}
	where, $G_k$ is the function defined in the previous paragraph. The notation \\  $\{ x_1,\dots,\tilde{x_l},\dots,x_k\} $ stands for $ \{ x_1,\dots,x_{l-1},x_{l+1},\dots,x_k\} $.
\end{prop}
\begin{proof}
	Recall that by definition,
	\begin{equation*}
		\E[ F \big( \delta(u)  \big)^n ] = \E[ F \big(   \int v(y,\xi) \, \nu(\d y) \big)^n ] ,
	\end{equation*}
	where $v(y,\xi) =  u(y,\xi \moins y)$, for $(y,\xi) \in E \times \X$. Therefore, we can apply proposition \ref{prop2}, and we obtain the desired result.
\end{proof}

The previous definition of the operator $\delta$ is justified by the duality formula with respect to the difference gradient, defined below.
\begin{defin}[Difference operator]
	For $F : \X \rightarrow \R$, we define $D F$, the difference operator applied to $F$, as follows:
	\begin{equation*}
		\left.  \begin{array}{l l l l}
		 D F: & E \times \X  &\longrightarrow \ &\R \\ 
		& (x,\xi)  &\longmapsto & D_x F(\xi) = F(\xi \cup x) - F(\xi \moins x)	.
		\end{array}\right.
	\end{equation*}
\end{defin}
Then, the previous definition of $\delta$ satisfies the following corollary of proposition \ref{prop4}.
\begin{cor}[Duality relation]
	For any bounded function $F$ on $\X$, and a process $u : (x,\xi) \mapsto u(x,\xi)$  in ${\rm Dom } (\delta)$, we have
	\begin{equation*}
		\E[ F \delta(u) ] = \E[ \int_E D_z F(\xi) u(z,\xi) c(z,\xi) \, \lambda(\d z)  ].
	\end{equation*}
\end{cor}

Here, we say that a process $u : (x,\xi) \mapsto u(x,\xi)$ belongs to $\dom( \delta)$ whenever there exists a constant $c$ such that for any $F\in L^2(\mu)$, 
\begin{equation*}
	\Big| \int_{\X} \int_E D_z F(\xi) u(z,\xi) \,C_\mu(\d z,\d \xi) \Big| \le c \|F \|_{L^2(\mu)}.
\end{equation*}
Recall that in the Poisson case, $C_\mu = \lambda \otimes \mu$ and we recover the classical duality relationship.

One also finds an extended Skorohod isometry in the case of a more general point process (see \cite{DVb}):
\begin{cor}
\label{cor1}
	\begin{align*}
		\E[ &\delta(u)^2 ] = \E[ \int u(y,\xi)^2 c(y,\xi) \, \lambda(\d y) ] \\
			&+ \E[ \int \int u(y,\xi \cup z) \big(u(z,\xi \cup y) - 2 u(z,\xi) \big) \hat{c}(\{ y,z \},\xi) \, \lambda(\d y) \lambda(\d z) ] \\
			&+ \E[ \big( \int u(y,\xi) c(y,\xi) \, \lambda( \d y) \big)^2 ].
	\end{align*}
\end{cor}

\section{Random transformation of the point process}

\subsection{Exvisibility}

The goal of this section is to give an application of the previous moment formulae. As was explained in the introduction, we wish to study a random transformation of the point process measure $\mu$. The main condition that enables us to study the random transformation in depth is that of exvisibility (to our knowledge, this notion was first introduced in \cite{DVb}). We will start by recalling the definition of exvisibility as well as most of the main properties.
\begin{defin}[Exvisibility]
	Define the exvisible $\sigma$-algebra $\mathcal{Z}$ to be the $\sigma$-algebra generated by sets of the form $B \times U$,  $B \in \mathcal{B}$, $U \in \mathcal{F}_{B^c}$. A stochastic process is said to be exvisible if it is measurable with respect to $\mathcal{Z}$ on $E \times \X$.
\end{defin}
In particular, let us introduce what we will call simple exvisible processes:
\begin{defin}[Simple exvisible process]
	We call simple exvisible process a stochastic process of the form
	\begin{equation*}
		u(x,\xi) = \sum_{i=1}^N \1{A_i} (x) F_i(\xi)	,
	\end{equation*}
	where $N \in \N$, $A_1,\dots,A_N \in \mathcal{B}$ and the $F_i$ are $\F_{A_i^c}$-measurable for $i=1,\dots,N$.
\end{defin}
This notion is important because of the following proposition:
\begin{prop}
	A stochastic process is exvisible if and only if it is the limit of simple exvisible processes, i.e. there exists a sequence $(A_k^n)$ of subdivisions of $E$, and $\F_{A_i^c}$-measurable $(F_k^n)$ such that
	\begin{equation*}
		\forall x \in E, \ \forall \xi \in \X, \quad u(x,\xi) = \lim_n \sum_k \1{A_k^n}(x) F_k^n(\xi_{(A_k^n)^c}).
	\end{equation*}
\end{prop}

Let us also recall the following property of exvisible processes.
\begin{prop}
	For $u$ an exvisible process, we have
	\begin{equation*}
		\forall x \in E, \ \forall \xi \in \X, \quad D_x u(x,\xi) = 0	.
	\end{equation*}
\end{prop}
\begin{proof}
		Assume that $u$ is of the type $u (x,\xi) = \1{A}(x)F(\xi_{A^c})$, where $A \in \mathcal{B}$ and $F$ is a measurable function on $\X$. We can write
		\begin{equation*}
			D_x u(x,\xi) = \1A(x) \Big(F((\xi \cup x)_{A^c}) - F(\xi_{A^c})\Big) =0	,
		\end{equation*}
		since $(\xi \cup x)_{A^c} = \xi_{A^c}$, when $x \in A$. The result then follows from the monotone class theorem.
\end{proof}
The following proposition will also be of particular interest:
\begin{prop}
\label{prop3}
	Let $u:E\times \X \rightarrow \R$ be an exvisible process. Then, it holds that
	\begin{equation*}
		D_{x_1} u(x_2,\xi) \dots D_{x_{k-1}} u(x_k,\xi)D_{x_k} u(x_1,\xi) = 0,
	\end{equation*}
	for $k \ge 2$,  $(x_1,\dots,x_k) \in E$ and $\xi \in \X$. More generally, we can easily prove that for $k \ge 2$,
	\begin{equation*}
		\prod_{l=1}^k u(x_l, \xi \cup x_1 \cup  \dots \cup x_k)^{n_l} = \prod_{l=1}^k u(x_l, \xi)^{n_l},
	\end{equation*}
	for all $(x_1,\dots,x_k) \in E$, $\xi \in \X$, and $(n_1,\dots,n_k) \in \N^k$.
\end{prop}
\begin{proof}
	We will start by proving the first part. Let us consider an exvisible process $u$, as well as $k \ge 2$,  $(x_1,\dots,x_k) \in E$ and $\xi \in \X$. Then,
	\begin{equation*}
		\prod_{l=1}^k D_{x_l} u({x_{l+1}}, \xi) = \sum_{I,J \subset \{1,\dots,k \}  ,\ I \cap J = \emptyset }  \quad  \prod_{i \in I } u(x_{i+1},\xi \cup x_{i} ) \prod_{j \in J} (-u(x_{j+1}, \xi) ),
	\end{equation*}
	where we have used the convention $x_{k+1} = x_1$. Now, if we assume that $u$ is of the form $u(x,\xi) =  \1{A}(x)F(\xi_{A^c})$, we have
	\begin{align*}
		\prod_{l=1}^k D_{x_l} u(x_{l+1}, \xi) &= \sum_{I,J \subset \{1,\dots,k \}  ,\ I \cap J = \emptyset } \! \! \! \! \! \! \! \! \1{A^k}(x_1,\dots,x_k) \prod_{i \in I } F((\xi\cup x_{i} )_{A^c}) \prod_{j \in J} (-F(\xi_{A^c}) )\\
		&= \sum_{I,J \subset \{1,\dots,k \}  ,\ I \cap J = \emptyset } \! \! \! \! \! \! \! \! \1{A^k}(x_1,\dots,x_k) \prod_{i \in I } F(\xi_{A^c}) \prod_{j \in J} (-F(\xi_{A^c}) )\\
		&= 0.
	\end{align*}
	Hence, the same holds for any exvisible $u$ by the monotone class theorem. The generalization follows by the same arguments.
\end{proof}
Exvisibility is in fact similar to adaptedness for processes on the real line. We can even reformulate the Skorohod isometry (corollary \ref{cor1}) in order to draw the parallel between point processes and processes on the real line :
\begin{align*}
		\E[ \delta(u)^2 ] = \E[ \int &u(y,\xi)^2 c(y,\xi) \, \lambda(\d y) ] \\
			&+ \E[ \int D_z u(y,\xi) D_y u(z,\xi)\hat{c}(\{ x,y \},\xi) \, \lambda(\d y) \lambda(\d z) ] \\
			&- \E[ \int \int u(z,\xi)  u(y,\xi) c(z,\xi) D_z c(y,\xi) \, \lambda( \d y) \lambda( \d z)].
\end{align*}
This particular way of writing the formula is useful since when $u$ is exvisible, $D_z u(y,\xi) D_y u(z,\xi) = 0$ for a.e. $x,y \in E$ and $\xi \in \X$, and therefore the second term is zero. Moreover, one would have the expected Skorohod isometry $\E[ \delta(u)^2 ] = \E[ \int u(y,\xi)^2 c(y,\xi) \, \lambda(\d y) ] $ if and only if $c(y,\xi) = c(y,\xi \cup z) $ for a.e. $(y,z) \in E^2$ and a.e. $\xi \in \X$, i.e. if $c$ does not depend on the configuration $\xi$. Hence, $\mu$ would therefore necessarily be a PPP in this case (see \cite{Ma}).

\subsection{Random transformations}

Now, let us consider a random shifting $\tau : E \times \X \rightarrow E$. For $\xi \in \X$, consider the image measure of $\xi$ by $\tau$, which we will call the random transformation $\tau_* (\xi)$, defined as
\begin{equation*}
	\tau_* (\xi) = \sum_{x \in \xi} \delta_{\tau(x,\xi)},
\end{equation*}
and thus $\tau_*$ shifts each point of the configuration in the direction $\tau$. We will say that $\tau_*$ is exvisible if $\tau$ is.
Now, we wish to study the effect of the the transformation on the underlying measure $\mu$ under sufficiently strong conditions on $\tau$. The following hypotheses will be considered:
\begin{description}
	\item[(H1)] The random transformation $\tau_*$ is exvisible, in the sense defined previously in this section.
	\item[(H2)] For a.e. $\xi \in \X$, $\tau(\cdot, \xi)$ is invertible, and we will note its inverse $\tau^{-1}(x,\xi)$, $x \in E, \xi \in \X$. We will also denote by $\tau^{-1}_*(\xi)$ the image measure of $\xi$ by $\tau^{-1}$.
\end{description}
\begin{thm}
\label{tau}
	Let $\tau : E \times \X \rightarrow E$ be a random shifting as defined previously, and satisfying ${ \bf (H1)}$ and  ${ \bf (H2)}$. Let us assume that $\tau$ maps $\lambda$ to $\sigma$, i.e. $\tau(\cdot, \xi) \lambda = \sigma$, $\xi \in \X$, where $\sigma$ is a fixed measure on $(E,\mathcal{B})$. Then, $\tau_* \mu$ has correlation functions
	\begin{equation}
	\label{cor}
		\rho_\tau(x_1,\dots,x_k) = \E \big[  \hat{c}(\{\tau^{-1}(x_1,\xi),\dots,\tau^{-1}(x_k,\xi) \},\xi) \big],\quad x_1,\dots,x_k \in E,
	\end{equation}
	with respect to $\sigma$.
\end{thm}
\begin{proof}
	Our aim for the proof is to use proposition \ref{charact} (characterization of correlation functions). By proposition \ref{prop1}, we have for any $n \in \N$, any measurable nonnegative (non-random) functions $v_k : E \rightarrow \R$, $k = 1,\dots,n$, we have
	\begin{multline*}
		\E_{\tau_* \mu}[\prod_{k=1}^n \int v_k(y) \,\xi(\d y) ] = \sum_{k=1}^n \sum_{\{P_1,\dots,P_k\} \in \mathcal{T}_n^k} \E_\mu \Big[ \int_{E^k}  \hat{c}(\{x_1,\dots,x_k \},\xi) \\
			\shoveright{\prod_{l=1}^k \prod_{i \in P_l} v_i (\tau(x_l,\xi \cup x_1 \cup \dots \cup x_k)) \,  \lambda(\d x_1)\dots \lambda(\d x_k)  \Big] } \\
			\shoveleft{= \sum_{k=1}^n \sum_{\{P_1,\dots,P_k\} \in \mathcal{T}_n^k} \E_\mu \Big[ \int_{E^k}  \hat{c}(\{x_1,\dots,x_k \},\xi)} \\
			\shoveright{\prod_{l=1}^k \prod_{i \in P_l} v_i (\tau(x_l,\xi)) \,  \lambda(\d x_1)\dots \lambda(\d x_k)  \Big]} \\
			\shoveleft{= \sum_{k=1}^n \sum_{\{P_1,\dots,P_k\} \in \mathcal{T}_n^k} \E_\mu \Big[ \int_{E^k}  \hat{c}(\{\tau^{-1}(x_1,\xi),\dots,\tau^{-1}(x_k,\xi) \},\xi)} \\
			\shoveright{ \prod_{l=1}^k \prod_{i \in P_l} v_i (x_l) \,  \sigma(\d x_1)\dots \sigma(\d x_k)  \Big],}\\
	\end{multline*}
	where the second equality follows from ${ \bf (H1)}$ and proposition \ref{prop2}. By proposition \ref{charact}, we conclude that $\tau_* \mu$ has correlation functions (with respect to $\sigma$), which are given by 
	\begin{equation*}
		\rho_\tau(x_1,\dots,x_k) = \E \big[  \hat{c}(\{\tau^{-1}(x_1,\xi),\dots,\tau^{-1}(x_k,\xi) \},\xi) \big],\quad x_1,\dots,x_k \in E,
	\end{equation*}
	since the previous functions are obviously symmetrical (by symmetry of $\hat{c}(\cdot, \xi)$ as a function on $E^n$), and since an invertible transformation of a simple point process $\mu$ yields another simple process (allowing us to apply proposition \ref{charact}).
\end{proof}
This theorem directly generalizes all known results. Indeed, consider the following corollary:
\begin{cor}
	Let $\mu=\pi^{ \d \lambda}$ be the Poisson measure with intensity $\lambda$. Let $\tau : E \times \X \rightarrow E$ be a random transformation satisfying ${ \bf (H1)}$ and  ${ \bf (H2)}$. Let us assume that $\tau$ maps $\lambda$ to $\sigma$, i.e. $\tau(\cdot, \xi) \lambda = \sigma$, $\xi \in \X$. Then, $\tau$ maps $\pi^{ \d \lambda}$ to $\pi^{ \d \sigma}$.
\end{cor}
\begin{proof}
	The corollary follows directly from the theorem, since $\pi^{ \d \lambda}$ has a Papangelou intensity of $1$. Therefore, $\tau_* \pi^{ \d \lambda}$ has intensity $\sigma$ and its correlation functions, given by (\ref{cor}), are also equal to $1$.
\end{proof}
We also recover some previously known results of quasi-invariance in the case of the determinantal random measure. Let us start by introducing determinantal point processes, and recalling some well-known facts. The interested reader can find further results concerning determinantal processes in \cite{HKPa, STa}. We define a Hilbert-Schmidt operator $K$ from $L^2(E,\lambda)$ to $L^2(E,\lambda)$ satisfying the following conditions:
\begin{description}
	\item[(i)] $K$ is a bounded symmetric integral operator on $L^2(E,\lambda)$, and we also write $K(\cdot,\cdot)$ for its kernel.
	\item[(ii)] The spectrum of $K$ is included in $[0,1[$ (Here, we exclude $1$ in order to ensure the existence of Papangelou intensities).
	\item[(iii)] $K$ is of locally trace class.
\end{description}
Under these conditions, we define $(\mu_K,\lambda)$ to be the (unique) measure with intensity $\lambda$ and correlation functions $\rho(x_1,\dots,x_n) = \det \big(K (x_i,x_j) \big)_{1 \le i,j \le n}$ (existence and uniqueness of point processes with given correlation functions are discussed in \cite{Lc,Ld}, see as well \cite{Sa} for the case of determinantal processes). We further define the operator $J$, called the global interaction operator, as $J :=K (I-K)^{-1}$. Since we assumed that $||K || < 1$, $J$ is properly defined as a bounded operator. We will also define $J_{[\Lambda]} := K_\Lambda (I - K_\Lambda)^{-1}$, where we will be wary of the fact that $J_{[\Lambda]}$ is not the restriction of $J$ to $\Lambda$, but rather $J$ applied to the operator $K_\Lambda$.
In order to ensure the existence of the Papangelou intensities, let us introduce the following condition on the operator $J$:
\begin{description}
	\item[(H3)] \
	\begin{itemize}
		\item $J$ has a continuous integral kernel.
		\item $J$ has finite range $R < \infty$.
		\item $\mu$ does not percolate.
	\end{itemize}
\end{description}
Here, ${ \bf (H3)}$ was the condition introduced in \cite{GYa} to ensure the existence of Papangelou intensities for the point process once we are not limited to the compact $\Lambda$. More precisely, it was shown that under ${ \bf (H3)}$, the determinantal process with global interaction operator $J$ satisfies condition $(\Sigma_{\lambda})$, i.e. it admits Papangelou intensities $c_K$. These intensities are well defined for the measure $\mu_{K,\Lambda}$ (restriction of $\mu$ to $\Lambda$) as $c_{K, \Lambda}(x,\xi) = \frac{\det J_{[\Lambda]} (\xi \cup x)}{\det J_{[\Lambda]} (\xi)}$, $x \in E, \xi \in \X$. Here, the notation $J_{[\Lambda]} (\xi)$ stands for $\big(J_{[\Lambda]} (x_i,x_j) \big)_{1 \le i,j \le n}$ where $\xi = \{x_1,\dots,x_n \} \in \X$. However, the main result of \cite{GYa} is that $(\mu_K,\lambda)$ also satisfies condition $(\Sigma_{\lambda})$ under condition ${ \bf (H3)}$, and $c_K$ is explicitly given by the formula:
\begin{equation*}
	c_K(x,\xi) =  \frac{\det J (\xi_W \cup x)}{\det J (\xi_W)} \1{\text{diam } W(x,\xi) < \infty},
\end{equation*}
where $W(x,\xi)$ is the union of the clusters of $B_R(x \cup \xi)$ hitting $x$ and $\xi_W := \xi_{W(x,\xi)}$.
With these preliminaries in mind, we can apply our previous results.
\begin{cor}
	Let $\mu_K$ be the determinantal measure with intensity $\lambda$ and kernel $K$. Assume that the associated operator $J$ satisfies ${ \bf (H3)}$. Let $\tau : E \times \X \rightarrow E$ be a transformation satisfying ${ \bf (H1)}$ and  ${ \bf (H2)}$. Let us assume that $\tau$ maps $\lambda$ to $\sigma$, i.e. $\tau(\cdot,\xi) \lambda = \sigma$, $\xi \in \X$. Then, $\tau_* \mu_K$ has correlation functions given by 
	\begin{equation}
		\rho_\tau(x_1,\dots,x_k) = \E \big[  \frac{\det J (\xi_W \cup \{\tau^{-1}(x_1,\xi),\dots,\tau^{-1}(x_k,\xi) \})}{\det J (\xi_W)}  \big],\quad x_1,\dots,x_k \in E,
	\end{equation}
	 and where $W$ is the union of the clusters of $B_R( \{\tau^{-1}(x_1,\xi),\hdots,\tau^{-1}(x_k,\xi) \} \cup \xi)$ hitting $ \{\tau^{-1}(x_1,\xi),\dots,\tau^{-1}(x_k,\xi) \}$. Moreover, if we further assume that $\tau$ is a non-random, invertible transformation, and define
	\begin{equation*}
		\left.  \begin{array}{l l}
		 T: &L^2(\sigma) \longrightarrow L^2(\lambda)	, \\ 
		&f \ \longmapsto f \circ \tau	.
	\end{array}\right.
	\end{equation*}
	Then, $\tau$ maps $(\mu_K,\lambda)$ to $(\mu_{K_\tau},\sigma)$, where $K_\tau = T^{-1} K T$.
\end{cor}
\begin{proof}
	The first part of the corollary is a direct consequence of theorem \ref{tau}. If we further assume that $\tau$ is non-random, then by (\ref{evc}), we have 
	\begin{equation*}
		\rho_\tau(x_1,\dots,x_k) = \rho(\tau^{-1}(x_1),\dots,\tau^{-1}(x_n)) = \det \big(K (\tau^{-1} (x_i),\tau^{-1}(x_j)) \big)_{1 \le i,j \le n}.
	\end{equation*}
	Its then remains to notice that the kernel $K_\tau = T^{-1} K T$ is again an integral operator with kernel $K_\tau (x,y)= K(\tau^{-1}(x),\tau^{-1}(y)) $, $x,y \in E$.
\end{proof}
\begin{rem}
	If we do not assume that $\tau$ is non-random, then there is no reason for $\rho_\tau$ to be given by the determinant of a Hilbert-Schmidt operator $K'$. Therefore, $\tau_* \mu_K$ is not necessarily determinantal in the general case of a random shift $\tau$.
\end{rem}
We can note that the last part of the corollary is another formulation of the quasi-invariance results obtained in \cite{CDa}. The study in the aforementioned paper was limited to determinantal processes on a compact $\Lambda$. On such a compact, ${ \bf (H3)}$ is obviously satisfied and we therefore have the existence of Papangelou intensities.

\bibliographystyle{hamsalpha}
\bibliography{mybib}

\end{document}